 \documentclass[11pt,letterpaper]{amsart}

 \usepackage{txfonts} %%%%%%%%字体 $\lambda$

 \usepackage[colorlinks,linkcolor=blue,citecolor=blue]{hyperref}

 \theoremstyle{plain}
 
\newtheorem{theorem}{Theorem}[section]
\newtheorem{lemma}[theorem]{Lemma}

\theoremstyle{definition}

\theoremstyle{remark}
\newtheorem{remark}[theorem]{Remark}

\numberwithin{equation}{section}

\begin{document}

%\title[Partial uniform ellipticity and prescribed curvature problem]{The partial uniform ellipticity and prescribed problems on the conformal classes of complete metrics}

\title[Conformally bending a closed manifold]{On the conformal bending of a closed Riemannian manifold}


%    author one information
% \author[short version for running head]{Rirong Yuan}
\author{Rirong Yuan }
\address{School of Mathematics, 
	South China University of Technology, 
	Guangzhou 510641, China}
%\curraddr{}
\email{yuanrr@scut.edu.cn}
 %\thanks{Research supported in part by NSFC (Grant No. 11801587).}
 %\thanks{The author is supported by the NSFC   Grant No. 11801587.}

%    \subjclass is required.
 % \subjclass[2010]{35J15, 53C21, 53A30, 58J05, 45J60.}
%\subjclass[2020]{}
 % \keywords{Fully nonlinear equation; partial uniform ellipticity; fully nonlinear Loewner-Nirenberg problem; complete noncompact version of fully nonlinear Yamabe problem}    

\date{}

\dedicatory{}

\begin{abstract}
	
In this paper, we bend a closed Riemannian manifold in the conformal class, through solving a fully nonlinear equation. As a result, we prove that each metric of quasi-negative Ricci curvature is conformal to a metric with negative Ricci curvature.
	
\end{abstract}

\maketitle

   %\tableofcontents

  %\setcounter{section}{-1}
  
  \section{Introduction}

    \medskip
    
    In 1980s, Gao-Yau \cite{Gao1986Yau} proved that any closed three-manifold admits a metric with negative Ricci curvature. 
  %A different proof of Gao-Yau's result was given by 
 %Subsequently, Brooks \cite{Brooks1989} gave a different proof of Gao-Yau's result.
The existence of negatively Ricci curved metric was extended by Lohkamp \cite{Lohkamp-1} to higher dimensional closed manifolds.   
    
%Inspired by their theorems, 
One natural question to ask is: 
   %let $M$ be a closed manifold of dimension $n\geq 3$
 given a conformal class $[g]=\{e^{2u}g: u\in C^\infty(M)\}$, is there a metric with negative Ricci curvature.

   When $M$ is an open manifold, this problem has been studied by Lohkamp \cite{Lohkamp-2}, who proved that every Riemannian metric on an open manifold is conformal to a complete metric of negative Ricci curvature. A new proof was given by the author in \cite{yuan-PUE1}. 
   %using fully nonlinear equations.
   
  Nevertheless,   the problem is still open when $M$ is closed.
 In this paper, we answer the problem when the Ricci curvature of $g$
is quasi-negative.

   First we summarize some notations and notions.
%Let $(M,g)$ be a closed Riemannian manifold of dimension $n\geqslant3$, 
Let $Ric_g$ be the Ricci curvature of Riemannian metric $g$.
 The $Ric_g$ is quasi-negative means that $Ric_g$ is nonpositive everywhere but strictly negative somewhere.  
 
 \begin{theorem}
 	\label{thm1-Ricci}
 	Let $(M,g)$ be a closed connected Riemannian manifold of dimension $n\geqslant 3$ with quasi-negative Ricci curvature. Then there is a unique smooth Riemannian metric $\tilde{g}\in [g]$ with  
 	\begin{equation}
 		\begin{aligned}    \det(\lambda(-{\tilde{g}^{-1}}Ric_{\tilde{g}}))=1. \nonumber
 			\end{aligned}
 	\end{equation}
 %	where $\sigma_k$ is the $k$-th elementary symmetric function.
 	\end{theorem}
 
% As a corollary, we obtain 
%\begin{corollary}	Let $(M,g)$ be a closed connected Riemannian manifold of dimension $n\geqslant 3$ with nonpositive Ricci curvature $Ric_g\leqslant 0$. Then either
	%	\begin{itemize}
			%\item $g$ is Ricci flat ($Ric_g=0$), or
			
			%\item $g$ is   conformal to a  metric of negative Ricci curvature.
	%	\end{itemize}
%\end{corollary}

  In fact we prove a more general theorem for a modified Schouten tensor 
  \begin{equation}
  	\begin{aligned}
  		\,& A_{{g}}^{\tau,\alpha}=\frac{\alpha}{n-2} \left({Ric}_{g}-\frac{\tau}{2(n-1)}   {R}_{g}\cdot {g}\right), \,& \alpha=\pm1,  \mbox{  }\tau \in \mathbb{R},  \nonumber
  	\end{aligned}
  \end{equation}  
where ${R}_g$ is the scalar curvature of $g$.
 More precisely, 
we can construct a metric   $\tilde{g}=e^{2{u}}g$ satisfying $\lambda(g^{-1}A_{\tilde{g}}^{\tau,\alpha})\in\Gamma$ in $M$,
  through solving a fully nonlinear equation 
\begin{equation}
	\label{main-equ1}
	\begin{aligned}
		f(\lambda(\tilde{g}^{-1}A_{\tilde{g}}^{\tau,\alpha}))=\psi,
	\end{aligned}
\end{equation}
where 
$\lambda(\tilde{g}^{-1}A_{\tilde{g}}^{\tau,\alpha})$ are the eigenvalues of $A_{\tilde{g}}^{\tau,\alpha}$ with respect to $\tilde{g}$, $0<\psi\in C^\infty(M)$ and 
  \begin{equation}
	\label{tau-alpha}
	\begin{cases}
		\tau<1 \,&\mbox{ if } \alpha=-1, \\
		\tau>1+(n-2)(1-\kappa_\Gamma\vartheta_{\Gamma}) \,&\mbox{ if } \alpha=1.
	\end{cases}
\end{equation}  
Here   $\kappa_\Gamma$ and $\vartheta_\Gamma$ are the constants from 
Theorem \ref{yuan-k+1} below. 

As in
\cite{CNS3}, $f$ is a \textit{smooth, symmetric} and \textit{concave} function defined in an \textit{open, symmetric} and \textit{convex} cone $\Gamma\subset\mathbb{R}^n$ with  vertex at the origin and boundary $\partial \Gamma\neq \emptyset$,  containing the positive cone $\Gamma_n:=\left\{\lambda \in \mathbb{R}^n: \mbox{ each component } \lambda_i>0\right\}.$ 
We denote the closure of $\Gamma$ by $\overline{\Gamma}=\Gamma\cup\partial \Gamma$. 

Our main result can be stated as follows. 
\begin{theorem}
	\label{thm2-metric}
	Let $(\alpha,\tau)$ satisfy \eqref{tau-alpha}. 
	 Let $M$ be a closed connected   manifold of dimension $n\geqslant3$ and suppose a Riemannian metric $g$ with
	 \begin{equation}
	 	\label{assump1-metric}
	 	\begin{aligned}
	 		\lambda(g^{-1}A_{g}^{\tau,\alpha}) \in \overline{\Gamma} \mbox{ in } M,
	 	\end{aligned}
	 \end{equation}
	 \begin{equation}
	 	\label{assump2-metric}
	 	\begin{aligned}
	 		\lambda(g^{-1}A_{g}^{\tau,\alpha}) \in \Gamma \mbox{ at some } p_0\in M.
	 	\end{aligned}
	 \end{equation}
Assume in addition that
 \begin{equation}
 	\label{homogeneous-1-mu}
 	\begin{aligned}
 		f(t\lambda)=t^{\mathrm{\varsigma}} f(\lambda), \mbox{ } \forall \lambda\in\Gamma, \mbox{} \forall t>0,
 		\mbox{  for some constant } 0<\mathrm{\varsigma} \leqslant1,
 	\end{aligned}
 \end{equation}
 \begin{equation}
 	\label{homogeneous-1-buchong2}
 	\begin{aligned} 
 		f >0 \mbox{ in } \Gamma, \quad f =0 \mbox{ on } \partial \Gamma.
 	\end{aligned}
 \end{equation} 
Then there is a unique smooth admissible metric   $\tilde{g}\in [g]$ satisfying
\eqref{main-equ1}.
\end{theorem}

\begin{remark}
	For the equation \eqref{main-equ1}, we call $g$ an \textit{admissible} metric if   $\lambda(g^{-1}A_{{g}}^{\tau,\alpha})\in\Gamma$ in $M$. Meanwhile, we say $ {g}$ is \textit{weakly admissible} if $ {g}\in C^2$ and $\lambda( {g}^{-1}A_{ {g}}^{\tau,\alpha}) \in \overline{\Gamma} \mbox{ in } M$.
\end{remark}
 
  The equation  \eqref{main-equ1}  
 %can be viewed as a special case of    \eqref{mainequ-1general}  
 include many important equations as special cases.
 When $f=\sigma_1$, $\tau=0$ and %$\psi=+1$, $0$ or $-1$,
 $\psi$ is a proper constant, 
 it is closely related to the well-known Yamabe problem, proved by Schoen \cite{Schoen1984}
 %combining the important  work  of 
 with important contributions from 
 Aubin \cite{Aubin1976} and Trudinger \cite{Trudinger1968}. 
 For $f=\sigma_k^{1/k}$, $\psi=1$ and $\tau=\alpha=1$, it was proposed by Viaclovsky \cite{Viaclovsky2000},  referred to $k$-Yamabe problem. The $k$-Yamabe problem %has received  much attention in  
 has been much studied in recent years   \cite{ChangGurskyYang2002,Ge2006Wang,Guan2003Wang-CrelleJ,ABLi2003YYLi,Gursky2007Viaclovsky,ShengTrudingerWang2007}.  
 
 We shall mention more  known work \cite{Gursky2003Viaclovsky,Sheng2006Zhang} on prescribed $\sigma_k$ curvature equation
   \begin{equation}
 	\label{main-equ-sigmak}
 	\begin{aligned}
 		\sigma_k^{1/k}(\lambda(\tilde{g}^{-1}A_{\tilde{g}}^{\tau,\alpha}))=\psi
 	\end{aligned}
 \end{equation} 
 on a closed manifold, in which 
 \begin{equation}
 	\label{tau-alpha-2}
 	\begin{cases}
 		\tau<1 \,& \mbox{ if } \alpha=-1,\\
 		\tau>n-1 \,& \mbox{ if } \alpha=1.
 	\end{cases}
 \end{equation} 
%Under this assumption, the equation  is automatically of fully uniform ellipticity according to the formula \eqref{conformal-formula1} below.
% the  prescribed $\sigma_k$-curvature equation  on a closed manifold was studied by Gursky-Viaclovsky \cite{Gursky2003Viaclovsky}  with $\tau<1$, $\alpha=-1$, and by Sheng-Zhang \cite{Sheng2006Zhang} for $\tau>n-1$, $\alpha=1$ based on a flow approach. 
In  \cite{Gursky2003Viaclovsky} Gursky-Viaclovsky solved the equation for $\alpha=-1$ and $\tau<1$, 
%(see also \cite{LiJY2005Sheng} for a parabolic proof),
%subsequently  a parabolic proof was done by Li-Sheng \cite{LiJY2005Sheng}.  
 %the equation is automatically of uniform ellipticity, and 
while  the case $\tau>n-1$ and $\alpha=1$ was considered by  Sheng-Zhang \cite{Sheng2006Zhang}  via a fully nonlinear flow. 
Notice that their approach 
 relies crucially  upon two assumptions: one is the existence of admissible metric, the other  is  the assumption \eqref{tau-alpha-2}. This assumption automatically ensures the uniform  ellipticity (parabolicity) of the equation.

As a contrast our result and strategy are different. 
We want to stress that
 in Theorem \ref{thm2-metric} the  given metric allows to be weakly admissible. 
 To achieve this, we
 construct in Section \ref{sec16} an admissible metric, based on Morse theory and differential topology. 
  %As a result, we can show that every Riemannian metric with quasi-negative Ricci curvature can be conformally deformed to a metric of negative Ricci curvature. 
   In addition, the assumption \eqref{tau-alpha} is much more broader than \eqref{tau-alpha-2} and it allows the critical case $\tau=n-1$ (the Einstein tensor $G_g=Ric_{{g}}-\frac{R_{{g}}}{2}\cdot{g}$) when $\Gamma\neq\Gamma_n$. 
   To overcome the difficulty, we use Theorem \ref{yuan-k+1} to explore the structure of \eqref{main-equ1}.
  The critical case is fairly interesting in dimension three, since the Einstein tensor is closely related to sectional curvature.  See a formula in \cite[Section 2]{Gursky-Streets-Warren2010}. 
  %The fully nonlinear Loewner-Nirenberg problem and noncompact complete fully nonlinear Yamabe problem have been settled by the author \cite{yuan-PUE1}

 %\begin{theorem}	Let $(M,g)$ be a closed three-manifold with quasi-negative sectional curvature. Suppose, in addition to  \eqref{homogeneous-1-mu} and \eqref{homogeneous-1-buchong2}, that $\Gamma\neq\Gamma_3$. Then there is a unique smooth metric  $\tilde{g}\in[g]$ satisfying $f(\lambda(\tilde{g}^{-1}(G_{\tilde{g}})))=\psi, \lambda({g}^{-1}(G_{\tilde{g}}))\in\Gamma$. 	\end{theorem}

%The paper is organized as follows. In Section \ref{Sect-preliminaries} we first summarize some formula and then talk about the structure of operators

\medskip

  \section{Preliminaries and Notations}
  \label{Sect-preliminaries}
  
  \subsection{Some formulas and reduction of equation}
  
  By the formula under the conformal change $\tilde{g}=e^{2u}g$
  (see e.g. \cite{Besse1987}),   
  \begin{equation}
  	\begin{aligned}
  		{Ric}_{\tilde{g}}=\,& {Ric}_g -\Delta u g -(n-2)\nabla^2u-(n-2)|\nabla u|^2g +(n-2)du\otimes du. \nonumber
  	\end{aligned}
  \end{equation}
  %\begin{equation}	\begin{aligned}
  %		e^{2u}{R}_{\tilde{g}}=\,& {R}_g-2(n-1)\Delta u-(n-1)(n-2)|\nabla u|^2. \nonumber
  %	\end{aligned}\end{equation}
  Thus
  \begin{equation}
  	\label{conformal-formula1}
  	\begin{aligned}
  		A_{\tilde{g}}^{\tau,\alpha}
  		= A_{g}^{\tau,\alpha}
  		+\frac{\alpha(\tau-1)}{n-2}\Delta u g-\alpha  \nabla^2 u
  		+\frac{\alpha(\tau-2)}{2}|\nabla u|^2 g
  		% \\\,&
  		+\alpha  du\otimes du. %\nonumber
  	\end{aligned}
  \end{equation}
  
  We denote
  \begin{equation}
  	\label{beta-gamma-A2}
  	\begin{aligned}
  		V[u]=\Delta u g -\varrho\nabla^2 u+\gamma |\nabla u|^2 g +\varrho du\otimes du+A,   %\nonumber
  	\end{aligned}
  \end{equation} 
  \begin{equation}
  	\label{beta-gamma-A3}
  	\begin{aligned}
  		\varrho=\frac{n-2}{\tau-1}, \mbox{ }
  		\gamma=\frac{(\tau-2)(n-2)}{2(\tau-1)}, \mbox{ } %\gamma_2=\frac{n-2}{\tau-1}, \mbox{ }
  		A=\frac{n-2}{\alpha(\tau-1)} A_{g}^{\tau,\alpha}. %\nonumber %\tilde{\psi}=\frac{(n-2)\psi}{\alpha(\tau-1)}.
  	\end{aligned}
  \end{equation}
  Notice  that $V[u]=\frac{n-2}{\alpha(\tau-1)} A^{\tau,\alpha}_{\tilde{g}}$.
  The equation \eqref{main-equ1}  
  reads as follows
  \begin{equation}
  	\label{mainequ-02-2-1}
  	\begin{aligned}
  		f(\lambda(g^{-1}V[u]))	=\frac{(n-2)^{\mathrm{\varsigma}}\psi}{\alpha^\mathrm{\varsigma}(\tau-1)^\mathrm{\varsigma}} e^{2 \mathrm{\varsigma} u}. %\mbox{ in } M,
  	\end{aligned}
  \end{equation}
  
  \subsection{Structure of operators}

In prequel \cite{yuan-PUE1} the author explored the structure of nonlinear operator subject to 
  \begin{equation}
	\label{addistruc}
	\begin{aligned}
		%\mbox{For each $\sigma<\sup_{\Gamma}f$ and } 
		\lim_{t\rightarrow +\infty}f(t\lambda)>f(\mu) \mbox{ for any } \lambda, \mbox{ }\mu\in \Gamma.
	\end{aligned}
\end{equation}
%First we summarize some related notation and notions.
We denote $f_{i}(\lambda)= %f_{\lambda_i}(\lambda)=
\frac{\partial f}{\partial \lambda_{i}}(\lambda).$ 
%The author obtained the following theorem.
 \begin{theorem}[{\cite{yuan-PUE1}}]
 	%[{\cite[Theorem 1.4]{yuan-PUE1}}]
	\label{yuan-k+1}
	%Let $f$ and $\Gamma$ be as above and $\kappa_\Gamma$ be as defined in Definition \ref{yuan-kappa}.
	Suppose  $(f,\Gamma)$ satisfies   \eqref{addistruc}.
	Then  \begin{equation}
		\label{elliptic-weak}
		\begin{aligned}
			f_{i}(\lambda) \geqslant 0  
			\mbox{ in } \Gamma, 	\mbox{  } \forall 1\leqslant i\leqslant n, \nonumber
		\end{aligned}
	\end{equation}  
and
	for any $ \lambda\in \Gamma$ 
	with $\lambda_1 \leqslant \cdots \leqslant\lambda_n$,
	\begin{equation}
		\begin{aligned}
			f_{{i}}(\lambda) \geqslant   \vartheta_{\Gamma} \sum_{j=1}^{n}f_j(\lambda)>0, \mbox{  } \forall 1\leqslant i\leqslant \kappa_\Gamma+1, \nonumber 
		\end{aligned}
	\end{equation}
where ${\kappa}_{\Gamma}=\max\left\{k: ({\overbrace{0,\cdots,0}^{k}},{\overbrace{1,\cdots, 1}^{n-k}})\in \Gamma \right\}$, 
%\begin{equation}\begin{aligned}	 {\kappa}_{\Gamma}=\max \left\{k: ({\overbrace{0,\cdots,0}^{k}},{\overbrace{1,\cdots, 1}^{n-k}})\in \Gamma \right\}; %\nonumber 
%\end{aligned}\end{equation}
as well as
$\vartheta_\Gamma=\frac{1}{n}$ for   $\Gamma=\Gamma_n$, and
\begin{equation}
	\label{theta-gamma}
	\vartheta_\Gamma= 
	\sup_{(-\alpha_1,\cdots,-\alpha_{\kappa_\Gamma}, \alpha_{\kappa_\Gamma+1},\cdots, \alpha_n)\in \Gamma;\mbox{ } \alpha_i>0}\frac{\alpha_1/n}{\sum_{i=\kappa_\Gamma+1}^n \alpha_i-\sum_{i=2}^{\kappa_\Gamma}\alpha_i} \mbox{ for } \Gamma\neq\Gamma_n.  \nonumber
\end{equation} 
\end{theorem}

Based on Theorem \ref{yuan-k+1}, the author confirmed the uniform ellipticity of  \eqref{mainequ-02-2-1}. Firstly by %$(\alpha,\tau)$ satisfies 
\eqref{tau-alpha},     	\begin{equation}		\label{assumption-4}  		\begin{aligned}	\varrho<\frac{1}{1-\kappa_\Gamma \vartheta_{\Gamma}} \mbox{ and } \varrho\neq 0.  \end{aligned}	\end{equation} 
From {\cite[Proposition 3.2]{yuan-PUE1}}, we know
  \begin{theorem}[{\cite{yuan-PUE1}}]
  	%[{\cite[Proposition 3.2]{yuan-PUE1}}]
  	\label{thm-localestimates}
  	Suppose \eqref{tau-alpha}, \eqref{homogeneous-1-mu} and \eqref{homogeneous-1-buchong2} hold.
  	Then the equation \eqref{mainequ-02-2-1} is of uniform ellipticity at any   solution $u$ with 
  	$\lambda(g^{-1}V[u])\in\Gamma$.
  \end{theorem}

The following lemma plays a key role in the construction of admissible metrics.
  
    \begin{lemma}[{\cite{yuan-PUE1}}]
    	%[{\cite[Corollary 3.4]{yuan-PUE1}}]
  	\label{coro3-ingamma}
  	For $\varrho$ satisfying \eqref{assumption-4},  $(1,\cdots,1,1-\varrho)\in \Gamma.$
  	%	\begin{equation}	\label{111varrho}	\begin{aligned}		(1,\cdots,1,1-\varrho)\in \Gamma.	\end{aligned}	\end{equation}
  \end{lemma}

    \medskip
  \section{Construction of admissible metrics and   proof of main results}
  \label{sec16}

  In this section, we construct admissible metric  in $[g]$ under the assumptions \eqref{assump1-metric} and \eqref{assump2-metric}.

\begin{theorem}%[Construction of admissible metric]
	\label{thm1-construction}
	Let $(\alpha,\tau)$ satisfy \eqref{tau-alpha}.
	Suppose $M$ is a closed connected manifold of dimension $n\geqslant 3$ and suppose a Riemannian metric $g$  satisfying \eqref{assump1-metric} and \eqref{assump2-metric}.  Then there exists an admissible metric in $[g]$.
\end{theorem}

\begin{proof}
	
	We use the notation denoted in \eqref{beta-gamma-A2} and \eqref{beta-gamma-A3}.   	For a $C^2$-smooth  function $w$ on $M$, we 
	denote the critical set by
	\begin{equation}
		\mathcal{C}(w)=\left\{x\in  M: dw(x)=0\right\}. \nonumber
	\end{equation}

	By the assumption \eqref{assump2-metric} and the openness of $\Gamma$, there is a uniform positive constant $r_0$
	such that  
	    	\begin{equation}
	    		\label{key1}
		\begin{aligned}
			\lambda(g^{-1}A)\in \Gamma \mbox{ in } \overline{B_{r_0}(p_0)}.
		\end{aligned}
	\end{equation}
	
Take a smooth Morse function $w$ with
  the critical set  
   $$\mathcal{C}(w)=\{p_1,\cdots, p_m, p_{m+1}\cdots p_{m+k}\}$$
among which $p_1,\cdots, p_m$ are all the critical points  being in $M\setminus \overline{B_{r_0/2}(p_0)}$. 
% ($k\geq 0$). 
Pick $q_1, \cdots, q_m\in {B_{r_0/2}(p_0)}$ but not the critical point of $w$. By the homogeneity lemma (see e.g. \cite{Milnor-1997}), one can find a diffeomorphism
$h: M\to M$, which is smoothly isotopic to the identity, such that 
\begin{itemize}
	\item $h(p_i)=q_i$, $1\leqslant i\leqslant m$.
	\item $h(p_i)=p_i$, $m+1\leqslant i\leqslant m+k$.
\end{itemize}
Then we obtain a   Morse function 
  \begin{equation}
  	\label{Morse1-construction}
	\begin{aligned} 
		v=w\circ h^{-1}.
	\end{aligned}
\end{equation} 
One can check that 
  \begin{equation}
  		\label{key2}
	\begin{aligned}
		\mathcal{C}(v)=\{q_1,\cdots, q_m, p_{m+1}\cdots p_{m+k}\}\subset \overline{B_{r_0/2}(p_0)}.
	\end{aligned}
\end{equation}

Next we complete the proof. Assume $v\leqslant -1$.   	Take $\underline{u}=e^{Nv},$ $\underline{g}=e^{2\underline{u}}g,$
then \begin{equation}
	\begin{aligned}
		V[\underline{u}]= A+
		N^2e^{Nv}\left((\Delta v g-\varrho\nabla^2v)/N+(1+\gamma e^{Nv})|\nabla v|^2 g+\varrho (e^{Nv}-1)dv\otimes dv\right). \nonumber
	\end{aligned}
\end{equation}

Notice that  
    	\begin{equation}
    		\label{key3}
	\begin{aligned}
	\,&\lambda(g^{-1} ((1+\gamma e^{Nv})|\nabla v|^2 g+\varrho (e^{Nv}-1)dv\otimes dv ))\\
	=\,& |\nabla v|^2 (1, \cdots, 1, 1-\varrho)+  e^{Nv} |\nabla v|^2 (\gamma,\cdots,\gamma,\gamma+\varrho).
	\end{aligned}
\end{equation} 

By Lemma \ref{coro3-ingamma} and the openness of $\Gamma$,  
    	\begin{equation}
    		\label{key4}
	\begin{aligned}
		(1, \cdots, 1, 1-\varrho)+e^{Nv}(\gamma,\cdots,\gamma,\gamma+\varrho)\in \Gamma  \mbox{ for  } N\gg1.
	\end{aligned}
\end{equation}

\noindent{\bf Case 1}: $x\in \overline{B_{r_0}(p_0)}$. By \eqref{key1} and the openness of $\Gamma$, 
\begin{equation}
	\begin{aligned}
		\lambda(g^{-1}(A+
		N e^{Nv} (\Delta v g-\varrho\nabla^2v)) )\in \Gamma.
	\end{aligned}
\end{equation}
Combining  \eqref{key3} and \eqref{key4}, 
\begin{equation}
	\begin{aligned}
		\lambda(g^{-1}V[\underline{u}])\in\Gamma \mbox{ in } \overline{B_{r_0}(p_0)}.  \nonumber
	\end{aligned}
\end{equation}

\noindent{\bf Case 2}:  $x\notin \overline{B_{r_0}(p_0)}$. By \eqref{key2} there is a uniform positive constant $m_0$ such that $|\nabla v|^2\geqslant m_0$ in $M\setminus\overline{B_{r_0}(p_0)}$. By \eqref{key3}, \eqref{key4}, \eqref{assump1-metric}, and the openness of $\Gamma$, 
\begin{equation}
	\begin{aligned}
		\lambda(g^{-1}V[\underline{u}])\in\Gamma \mbox{ in } M\setminus\overline{B_{r_0}(p_0)}.  \nonumber
	\end{aligned}
\end{equation}

This completes the proof.
\end{proof}

Using the admissible metric constructed 
%in Theorem \ref{thm1-construction}, 
above, together with Theorem \ref{thm-localestimates}, we immediately derive the $C^0$-estimate using maximum principle. 

Building on Theorem \ref{thm-localestimates}, the author derived interior estimates for \eqref{mainequ-02-2-1}. 
\begin{theorem}[\cite{yuan-PUE1}]
	\label{interior-2nd-2} 
	Let $B_r\subset M$ be a geodesic ball of radius $r$.
	Let $u\in C^4(B_{r})$ be a  solution with $\lambda(g^{-1}V[u])\in\Gamma$ to the equation  \eqref{mainequ-02-2-1}  in $B_r$.
	Assume  \eqref{tau-alpha}, \eqref{homogeneous-1-mu} and \eqref{homogeneous-1-buchong2} hold.
	Then  
	\begin{equation}
		\begin{aligned}
			\sup_{B_{{r}/{2}}}\, (|\nabla^2 u|+|\nabla u|^2) \leqslant C, \nonumber 
		\end{aligned}
	\end{equation}
	where $C$ depends on  $|u|_{C^0(B_r)}$, $r^{-1}$, 
	and other known data.
\end{theorem}

  %Combining with Theorem \ref{interior-2nd-2}, we obtain 
In conclusion, we obtain
    	\begin{equation}
  	\begin{aligned}
  		|\nabla^2 u|\leqslant C. \nonumber
  	\end{aligned}
  \end{equation}
Then we can prove Theorem \ref{thm2-metric} via a standard continuity method. 
%See e.g. \cite{Gursky2003Viaclovsky}. 

  \bigskip
  
  %\noindent{\bf Acknowledgements}.
 %\subsubsection*{Acknowledgements}    The author is supported by the National Natural Science Foundation of China through grant 11801587.

%\bigskip
% \noindent{\bf Acknowledgements}.
\subsubsection*{Acknowledgements} 
The author would like to thank Professor Yi Liu for helpful discussions on the Morse theory.
 The author is supported by  the National Natural Science Foundation of China grant 11801587.
%Research partially supported by the National Natural Science Foundation of China, grant 11801587.

%\bigskip
\bigskip

%\small
%\bibliographystyle{amsplain}

\end{document}